\documentclass [11pt,reqno]{amsart}
\usepackage {amsmath,amssymb,verbatim,geometry,color,esint}
\usepackage[all]{xy}
\usepackage{esint}
\usepackage{mathrsfs}
\usepackage[pdftex,hyperindex]{hyperref}
\usepackage[pdftex]{graphicx}
\usepackage{tikz}
\usetikzlibrary{math}
\usepackage{tikz-cd}
\usepackage{float}
\usepackage[justification=centering]{caption}

 \newcommand{\C}{\mathbb{C}}

\renewcommand{\P}{\mathbb{P}}
 
 \newcommand{\R}{\mathbb{R}}
 \newcommand{\Z}{\mathbb{Z}}

\DeclareMathOperator{\affspan}{affspan}

\DeclareMathOperator{\conv}{Conv}

\DeclareMathOperator{\MA}{MA}

\DeclareMathOperator{\Vol}{Vol}

\DeclareMathOperator{\Star}{Star}

\DeclareMathOperator{\GL}{GL}
\DeclareMathOperator{\SL}{SL}
\DeclareMathOperator{\On}{O}

\numberwithin{equation}{section}       

\newtheorem{prop} {Proposition} [section]
\newtheorem{thm}[prop] {Theorem} 
\newtheorem{defi}[prop] {Definition}
\newtheorem{lem}[prop] {Lemma}

\newtheorem{cor}[prop]{Corollary}
\newtheorem{prop-def}[prop]{Proposition-Definition}

\newtheorem{exam}[prop]{Example}
\newtheorem{rmk}[prop]{Remark}

\theoremstyle{remark}

\title{Duality of Hessian manifolds and optimal transport}
\date{\today}

\author{Jakob Hultgren}

\address{Dept of Mathematics and Mathematical   Statistics\\
  Umeå University\\ 
  901 87 Umeå, Sweden}
\email{jakob.hultgren@umu.se}

\begin{document}

\maketitle

{\centering\footnotesize Dedicated to Professor Bo Berndtsson on the occasion of his 70th birthday \par}

\begin{abstract}
    This is an expository paper describing how duality theory for Hessian manifolds provides a natural setting for optimal transport. We explain how this can be used to solve Monge-Ampère equations and survey recent results along these lines with applications to the SYZ-conjecture in mirror symmetry. 
\end{abstract}

\setcounter{tocdepth}{1}
\tableofcontents

\section{Introduction}
An affine manifold is a smooth manifold with a distinguished atlas whose transition functions are affine. These arise naturally in symplectic geometry as the base space of any Lagrangian bundle with compact connected fibers. In particular, this puts them at the center of the SYZ-conjecture in mirror symmetry, which predicts that Calabi-Yau manifolds close to large complex structure limits admits (singular) special Lagrangian torus fibrations. The base spaces of these special Lagrangian torus fibrations are expected to be spheres equipped with distinguished atlases as above (usually referred to as affine structures) extending to the complement of a co-dimension 2 subset.

A Hessian metric is a Riemannian metric on an affine manifold which locally (in coordinates of the distinguished atlas) can be expressed as the Hessian of a convex function. An affine manifold equipped with a Hessian metric is a Hessian manifold. When the transition functions in the distinguished atlas are volume preserving, the affine structure is said to be special. In this case, the manifold admits a well-defined real Monge-Ampère operator associating to each Hessian metric a Borel measure. Solutions to equations involving this operator are conjectured to form asymptotic models of the Kähler potentials of Ricci flat metrics on Calabi-Yau manifolds close to large complex structure limits. Moreover, constructing solutions to real Monge-Ampère equations on integral affine manifolds with singularities has been central in recent progress on the SYZ-conjecture initiated by Y. Li (see \cite{LiFermat, LiSYZ, HJMM, PS, LiToric, AH}) and arguably constitute one of the main obstacle to further progress along similar lines. 

Optimal transport is a collection of ideas surrounding an optimization problems introduced independently by Monge in the 18th century and Kantorovich in the 1940's. This problem is concerned with transporting sand/resources/fluid/probability distributions from one configuration (modeled by a \emph{source} probability measure) to another configuration (modeled by a \emph{target} probability measure). 
It is well known that in 'model cases' of optimal transport (when the source and target measures are defined on $\R^d$ and has convex support) solving the optimal transport problem with respect to the squared distance cost function is equivalent to solving a real Monge-Ampère equation (see Section~\ref{sec:OptimalTransport} for details). 

Similarly, part of the recent progress on the SYZ-conjecture in \cite{HJMM, LiToric, AH} relies on a variational approach to Monge-Ampère equations on the boundary of reflexive polytopes introduced in \cite{HJMM} which is related to a particular optimal transport problem. Moreover, \cite{HO19} uses a variational approach to solve Monge-Ampère equations on compact affine manifolds related to a different optimal transport problem. The aim of this expository paper is to explain how both these optimal transport problem, at least in a heuristic manner, arise from classical duality theory for Hessian manifolds. 


We will begin by recalling some basic facts about convex functions, Legendre transform and Monge-Ampère equations, in particular their weak interpretation in terms of Aleksandrov. Then we will give a very brief background to optimal transport, focusing on the central fact that optimal transport plans are (in a generalized sense) supported on graphs of gradients of convex functions. 

In Section~\ref{sec:AffineManifolds} we will provide some background to affine and Hessian manifolds, including examples and a structural theorem. In Section~\ref{sec:DualHessianStructures} we will explain how each Hessian manifold has a dual Hessian manifold. We will also compute this dual for two classes of Hessian manifolds: 
\begin{itemize}
\item General compact Hessian manifolds
\item Affine structures induced on dense subsets of the boundaries of reflexive polytopes
\end{itemize}
This is possibly the only part of this paper which is not readily available in the litterature. A crucial point is that given an affine manifold in one of the classes above, for reasonable choises of Hessian metrics the affine structure in the dual Hessian manifold is almost fixed. This allows for a variational approach where the Hessian metric varies but the affine structure in the dual Hessian manifold is fixed. 

In Section~\ref{sec:Pairings}, we will explain how completing this picture with a 'pairing like' object between $(X,\nabla)$ and $(X^\vee,\nabla^\vee)$ defines a optimal transport problem which coincides with the functionals used in \cite{HO19,HJMM,LiToric,AH}. Finally, in Section~\ref{sec:MA} we survey the results in \cite{HO19,HJMM,LiToric,AH} with applications to mirror symmetry.

An important part of this topic which is not directly in the scope of this paper is that affine manifolds conjecturally arise as the essential skeletons of Berkovich spaces associated to maximally degenerate families of Calabi-Yau manifolds. This setting allows for an abstractly defined analog of the Monge-Ampère operator: the Chambert-Loir measure \cite{CL06}. See \cite{BJ, BFJ, LiSYZ, OO21, CLD, BGJKM21, NX16, NXY, MP21} for a very incomplete list of the great progress recent decades has seen along these lines. In particular, the latter two addresses the question of how to assign an integral affine structure to the essential skeleton. Moreover, see \cite{LiOT} for an account of how optimal transport arise directly out of this framework and \cite{Lag} for an influential idea on how to set up a tropical pluripotential theory.  

Finally, the idea to use Legendre transform type constructions to set up a variational framework for Monge-Ampère equations goes back at least to Aleksandrov's solution of the Minkowski problem. More recently, the variational approach in \cite{HO19,HJMM, LiToric, AH} can be compared to the variational approach used to produce Kähler-Einstein metrics on $K$-stable toric Fano manifolds in \cite{BB, San}, which also plays a major role in \cite{BPerm,HPerm}, and the variational approach studied very recently in \cite{TongYau}. 

\subsection*{Acknowledgements}
The author would like to thank Rolf Andreasson, Mattias Jonsson, Enrica Mazzon, Nick McCleerey, Yang Li and Magnus Önnheim for many discussions on the subject. 

\section{Background}\label{sec:Background}
\subsection{Convex Functions}
Given an open convex subset $U$ of $\R^d$, a function $\Phi:U\rightarrow \R$ is convex if its epigraph $\{(x,h)\in U\times \R: h\geq \Phi(x) \}$ is convex. If $\Phi$ is smooth, then $\Phi$ is convex if and only if its Hessian is positive semi-definite. In the general case, convexity implies some strong regularity properties. In particular, any convex $\Phi:U\rightarrow \R$ is automatically locally Lipschitz continuous on $U$. 

The (multivalued) gradient map $\partial\Phi:U\rightarrow \R^d$ is defined by writing $y\in \partial\Phi(x)$ if 
$$  \Phi(x')\geq \Phi(x)+\langle x,y \rangle $$
for all $x'\in U$. Often it is useful to identify $\Phi$ with an $(\R\cup\{\infty\})$-valued extension to $\R^d$,  defined in the following way: Extend $\Phi$ by lower semi-continuity to the closure of $U$ and set $\Phi\equiv+\infty$ outside the closure of $U$. It can be checked that this extension does not affect $\partial\Phi(U)$. Conversely, if $\Phi:\R^d\rightarrow \{+\infty\}$ is a lower semi-continuous convex function, then $\{\Phi<+\infty\}$ is a convex subset whose interior correspond to $U$ above. A lower semi-continuous convex function $\Phi$ is differentiable at $x$ if and only if $\partial\Phi(x)$ consists of a single element. Any convex function $\Phi$ is differentiable almost everywhere on the interior of $\{\Phi<+\infty\}$. 

\subsection{Legendre Transform}
Given a lower semi-continuous function $\Phi:\R^d\rightarrow \R\cup \{\infty\}$, the Legendre transform $\Phi^*:(\R^d)^*\rightarrow \R\cup\{+\infty\}$ of $\Phi$ is the convex lower semi-continuous function
$$\Phi^*(y) = \sup_{x\in \R^d} \langle x,y \rangle - \Phi(x).$$
Similarly, given a lower semi-continuous function $\Psi:(\R^d)^*\rightarrow \R\cup \{+\infty\}$, the Legendre transform $\Psi^*:\R^d\rightarrow \R\cup\{+\infty\}$ of $\Psi$ is the convex lower semi-continuous function $\R^d\rightarrow \R\cup\{\infty\}$
$$\Psi^*(y) = \sup_{y\in (\R^d)^*} \langle x,y \rangle - \Psi(y).$$

An application of the Hahn-Banach supporting hyperplane theorem gives that 
$$ (\Phi^*)^* = \Phi \text{ and } (\Psi^*)^* = \Psi $$ 
if $\Phi:\R^d\rightarrow \R\cup\{+\infty\}$ and $\Psi:(\R^d)^*\rightarrow \R\cup\{+\infty\}$ are lower semi-continuous and convex. 

This duality is also manifested in the following ways.
\begin{itemize}
    \item If $\Phi$ is strictly convex and differentiable on the interior of $\{\Phi<+\infty\}$, then $\phi^*$ is strictly convex and differentiable on the interior of $\{\Phi<+\infty\}$ and 
    $$ \partial\phi^*\circ\partial\phi(x) = x $$
    $$ \partial\phi\circ\partial\phi^*(y) = y $$
    for all $x$ in the interior of $\{\Phi<+\infty\}$ and $y$ in the interior of $\{\Phi^*<+\infty\}$.
    \item More generally, if $\phi:R^d\rightarrow \R\cup\{+\infty\}$ is convex and 
    $(x,y)\in \R^d\times (\R^d)^*$,
    then 
    $$ y\in \partial \phi(x) \text{ if and only if } x\in \partial \phi^*(y). $$
    Moreover, this holds if and only if 
    $$ \phi(x)+\phi^*(y)=\langle x,y \rangle. $$
    \item If $\Phi$ is strictly convex and differentiable on the interior of $\{\Phi<+\infty\}$, then
    \begin{equation} \label{eq:HessLeg} (\phi^*)_{ij} = \phi^{ij}\circ \partial \phi^* \end{equation}
    on the interior of $\{\Phi^*<+\infty\}$.
\end{itemize}


\subsection{The Monge-Ampère Operator}
Let $\Phi:\R^d\rightarrow \R\cup\{+\infty\}$ be lower semi-continuous and assume $\Phi$ is smooth and strictly convex on the interior of $\{\Phi<+\infty\}$. The Monge-Ampère measure of $\phi$ on the interior of $\{ \Phi<+\infty \}$ is
$$ \MA(\phi) = \det\left(\frac{\partial^2\Phi}{\partial x_i\partial x_j}\right)dVol $$
where $dVol$ is the Lebesgue measure on $(\mathbb R^d)^*$. Since $\det\left(\frac{\partial^2\Phi}{\partial x_i\partial x_j}\right)$ is the Jacobian determinant of $\partial\Phi$, we get the following equality for all measurable subsets $A$ of the interior of $\{\Phi<+\infty\}$:
$$ \int_{\nabla\Phi(A)}dVol = \int_A \MA(\Phi). $$
This means that $\MA(\Phi)$ is the pullback of the Lebesgue measure on $\partial\Phi(\{\Phi<+\infty\}^\circ)$ by the diffeomorphism $\partial\phi$. Equivalently, $\MA(\Phi)$ is the pushforward of the Lebesgue measure on $\partial\Phi(\{\Phi<+\infty\}^\circ)$ by $\partial\Phi^*$. 

The latter description of $\MA(\Phi)$ generalises to lower semi-continuous convex functions $\Phi$, since $\Phi^*$ is convex and hence differentiable almost everywhere on the interior of $\{\Phi^*<+\infty\}$.
\begin{defi}[The Monge-Ampère operator in the sense of Aleksandrov]
Let $\Phi:\R^d\rightarrow \R\cup\{+\infty\}$ be lower semi-continuous and convex. The Monge-Ampère measure of $\Phi$ is
\begin{equation} \label{eq:MADefPF} \MA(\Phi) = (\partial\Phi^*)_\#d\Vol \end{equation}
where $dVol$ denotes the Lebesgue measure on the interior of $\{\Psi^*<+\infty\}\subset (\R^d)^*$. Equivalently, for every measurable set $A\subset \R^d$, 
\begin{equation} \label{eq:MADefGI} \MA(\Phi)(A) = \int\partial\Phi(A)dVol. \end{equation}
\end{defi}
\begin{rmk}
    Even though $\partial\Phi^*$ is a multivalued map, $\Phi^*$ is differentiable and $\partial\Phi^*$ single valued almost everywhere. Consequently, the push forward in \eqref{eq:MADefPF} is well-defined. To see that \eqref{eq:MADefPF} implies \eqref{eq:MADefGI}, note that by \eqref{eq:MADefPF}
    \begin{eqnarray*} 
     \MA(\Phi)(A) & = & |(\partial\Phi^*)^{-1}(A)| \\
     & = & |\{y: \Phi^* \text{ is differentiable at } y \text{ and }\partial\Phi^*(y)\in A\}| \\
     & = & |\{y: \Phi^* \text{ is differentiable at } y \text{ and } y\in \partial \Phi(A) \}|
     \end{eqnarray*}
     and that the difference between $\partial\Phi(A)$ and the set in the last row is contained in the non-differentiable locus of $\Phi^*$ which is of zero volume.
\end{rmk}

Note that $\MA(\Phi)=1$ on an open set $V\subset \{\Phi<+\infty\}$ if and only if the multivalued map $\partial\Phi$ is volume preserving. With a bit of work, this provides another nice duality principle:
$$ \MA(\Phi) = 1 $$
is satisfied on a set $V$ if and only if 
$$ \MA(\Phi^*) = 1. $$
on $\partial\Phi(V)\subset (\R^d)^*$. 

The Monge-Ampère operator is invariant under volume preserving affine transformations. The formula \eqref{eq:MADefGI} above provides a convenient way to see this: If $B\in \SL_d(\R), b\in \R^d$ and $\lambda$ is the affine transformation given by $x\rightarrow Bx+b$, then 
$$ |\partial(\Phi\circ \lambda)(A)| = |B^{-1}\partial\Phi(A)|=|\partial\Phi(A)| $$

See for example \cite{Figalli} for a more detailed account of the Monge-Ampère equation. 

\subsection{Optimal transport}\label{sec:OptimalTransport}
We will now give a brief account of basic optimal transport theory. For details, see for example \cite{AG,Villani}. Given two probability spaces $(Y,\mu)$ and $(Z,\nu)$ and a continuous \emph{cost function} $c:X\times Y\rightarrow \R$, optimal transport seeks to minimize the functional 
\begin{equation}
    \label{eq:OTPrimal}
    I(\gamma) = \int_{X\times Y} x\gamma
\end{equation} 
over all couplings $\gamma$ of $\mu$ and $\nu$, in other words all probability measures $\gamma$ on $X\times Y$ whose marginals are given by $\mu$ and $\nu$, respectively. The \emph{source measure} $\mu$ can be interpreted as a distributions of goods, and the \emph{target measure} as a distribution of consumers. A \emph{transport plan} $\gamma$ then describes a way to transport the goods to the consumers in the sense that $\gamma(A\times B)$ for two measurable subsets $A\subset X$ and $B\subset Y$ describes how much resources should be transported from $A$ to $B$. If we interpret $c(x,y)$ as the cost of transporting one unit of goods from $x\in X$ to $y\in Y$, then \eqref{eq:OTPrimal} is exactly the total cost of transporting the goods to the consumers according to the plan specified by $\gamma$. 

Under very mild conditions an \emph{optimal transport plan}, i.e a minimizer of \eqref{eq:OTPrimal}, exists. In particular, if $(X,\mu)$ and $(Y,\nu)$ are Polish probability spaces and $c$ continuous, then existence of an optimal transport plan follows from simple compactness arguments. However, there may be more than one optimal transport plan. 

\subsubsection{Dual Formulation}
As a linear optimization problem, optimal transport admits a dual formulation. This is given by maximizing
\begin{equation} \label{eq:OTDual} J(\Phi) = -\int \Phi \mu - \int \Phi^c \nu \end{equation}
over all $L^1(\mu)\ni\Phi:X\rightarrow \R$, where $\Phi^c:Y\rightarrow \R$ is the $c$-transform of $\Phi$ defined by 
$$ \Phi^c(y) = \sup c(x,y)-\Phi(x). $$
\begin{rmk}
    A symmetric version of \eqref{eq:OTDual} is normally used in the literature. This is given by maximizing 
    \begin{equation} \label{eq:OTDualSymmetric} -\int \Phi \mu - \int \Psi \nu \end{equation}
    over all $L^1(\mu)\ni\Phi:X\rightarrow \R$ and $L^1(\nu)\ni\Psi:Y\rightarrow \R$ such that $\Phi(x)+\Phi(y) \leq c(x,y)$ for all $x\in X$ and $y\in Y$. To get \eqref{eq:OTDual} from \eqref{eq:OTDualSymmetric}, simply note that if $\Phi$ is fixed, the best possible $\Psi$ is given by $\Psi^c$. Note also that the signs in \eqref{eq:OTDualSymmetric} and in the definition of $c$-transform differs from the conventional ones in the optimal transport litterature. 
\end{rmk}

We will give a heuristic explanation of how to arrive at \eqref{eq:OTDual} below. Before that we note that as usual the following inequality is satisfied by construction:
$$
\sup J(\Phi) \leq \min I(\gamma).
$$
Moreover, with some extra conditions, strong duality holds and the dual admits a maximizer:
\begin{thm}[\cite{Villani}[Theorem~5.10]]
    Let $(X,\mu),(Y,\nu)$ be polish spaces. Assume $c:X\times Y\rightarrow \R$ is continuous and bounded from below and that there is some transport plan of finite cost. Then $J$ admits a maximizer and 
    \begin{equation} \max J = \min I. \label{eq:StrongDuality} \end{equation}
\end{thm}
A striking feature of this is the following corollary 
\begin{cor}
A transport plan $\gamma$ is optimal if and only if it is supported on the set 
\begin{equation} \label{eq:Graph} \{(x,y)\in X\times Y:\Phi(x)+\Phi^c(y) = c(x,y) \}. \end{equation}
Moreover, if this holds, then $\Phi$ is a maximizer of $J$. 
\end{cor}
\begin{proof}
    By \eqref{eq:StrongDuality}, $\gamma$ is optimal if and only if $I(\gamma) = \max J = J(\Phi)$ for some $\Phi$. Moreover, for any $\Phi\in L^1(\mu)$
    \begin{eqnarray} 
        0 & \leq & I(\gamma)-J(\Phi) \label{eq:IEqualsJ}\\
        & = & \int c\gamma+\int\Phi\mu+\int\Phi^c\nu \\
        & = & \int (c(x,y)+\Phi(x)+\Phi^c(y)) \gamma.
    \end{eqnarray}
    Since $c(x,y)+\Phi(x)+\Phi^c(y)\geq 0$ and $\gamma\geq 0$, equality holds in \eqref{eq:IEqualsJ} if and only if $\gamma$ is supported on \eqref{eq:Graph}.
\end{proof}

\begin{cor}\label{cor:KnottSmith}
    Let $X=\mathbb R^n$, $Y$ a bounded convex subset of $(\R^d)^*$  and $c(x,y)=-\langle x,y \rangle$. Assume $\mu$ has finite second moment and $\nu$ is proportional to the Lebesgue measure on $Y$. Then $J(\Phi)=\max J$ if and only if there is a constant $\rho>0$ such that
    \begin{equation} 
    \MA(\Phi) = \rho\mu
    \end{equation}
    in the sense of Aleksandrov and $\overline{\partial\Phi(\R^d)} = \overline{Y}$.  
\end{cor}
\begin{rmk}
    The constant $\rho$ is picked so that the mass of $\rho\mu$ equals the volume of $Y$. 
\end{rmk}
\begin{rmk}
    When $\mu$ is absolutely continuous with positive density, Corollary~\ref{cor:KnottSmith} is essentially a special case of results due to Knott-Smith~\cite{KnottSmith} and Brenier~\cite{Brenier}. Identifying convexity as the cruzial assumption in order for $\Phi$ to define an Aleksandrov solution is due to Caffarelli \cite{Caf97}. 
\end{rmk}
\begin{proof}[Proof Sketch]
    Assume $\Phi$ is a minimizer. Since $\Phi=\Psi^c$ for some $\Psi:Y\rightarrow \R$ we have $\partial\Phi(\R^d)\subset Y$. Since $\partial\Phi^*$ is single valued almost everywhere with respect to $\nu$ and $\gamma$ is supported on the graph of $\partial\Phi^*$ it follows that $\gamma = (\partial\Phi^*,id)_\# \nu$ where $(\partial\Phi^*,id):Y\rightarrow X\times Y$ is the map given by $y\mapsto (\partial\Phi^*(y),y)$. Combining this with $(p_X)_\# \gamma = \mu$ we get $ (\partial\Phi^*)_\# \nu = \mu$, hence $\rho(\partial\Phi^*)_\# dVol = \mu$ for some $\rho>0$, which proves the corollary. 
\end{proof}

Before proceeding, we will give a brief indication on how $J$ arise as the dual of $I$. A general heuristic form formulating a dual optimization problem is to introduce Lagrange multipliers for each condition and consider the resulting unconstrained, but penalized, problem. We have three conditions: $\gamma>0$, $(p_X)_\#\gamma - \mu = 0$ and $(p_Y)_\#\gamma - \nu = 0$. 
Pick Lagrange multipliers $h>0,\Phi$ and $\Psi$. These are functions on $X\times Y$, $X$ and $Y$, respectively. 

The unconstrained problem is then 
\begin{eqnarray*} 
    \inf_\gamma L_{h,\Phi,\Psi}(\gamma) & := & \inf_\gamma \int_{X\times Y} c\gamma - \int_{X\times Y} h\gamma + \int_Y \Phi((p_X)_\#\gamma - \mu) + \int_Y \Psi((p_Y)_\#\gamma - \nu) \\
    & = & \inf_\gamma \int_{X\times Y} (c(x,y)+\Phi(x)+\Psi(y)-h(x,y)) \gamma - \int_X \Phi\mu - \int_Y \Psi\nu
\end{eqnarray*}
where the infimum is taken over all signed finite measures $\gamma$ on $X\times Y$. The dual optimization problem seeks to find $h,\Phi,\Psi$ which maximize the infimum of the unconstrained problem. It is clear that $\inf_\gamma L_{h,\Phi,\Psi}(\gamma) = -\infty$ unless $c(x,y)+\Phi(x)+\Psi(y)-h(x,y)=0$ for all $(x,y)\in X\times Y$. Fixing $h(x,y) = c(x,y)+\Phi(x)+\Psi(y)$, $L_{h,\Phi,\Psi}(\gamma)$ reduces to $\int \Phi\mu + \int \Psi\nu$ and since $h\geq 0$ we arrive at the dual problem 
$$ -\sup\int \Phi\mu - \int \Psi\nu $$
under the condition $c(x,y)\geq -\Phi(x)-\Psi(y)$ for all $(x,y)\in X\times Y$. 

Finally, $\Phi:X\rightarrow \R\cup\{+\infty\}$ and $\Psi:Y\rightarrow \R\cup\{+\infty\}$ are $c$-convex if $(\Phi^c)^c=\Phi$ and $(\Psi^c)^c=\Psi$.
Analogously to the classical multivalued gradient for convex functions, we define the $c$-gradient by $y\in \partial^c\Phi(x)$ and $x\in \partial^c\Phi^c$ if and only if 
$$ \Phi(x)+\Phi^*(y) = -c(x,y) $$
and similarly for $\Psi\in L^1(\nu)$. 

\section{Affine and Hessian Structures}\label{sec:AffineManifolds}
As explained in the introduction, an affine manifold is a smooth manifold equipped with an affine structure. We will give two equivalent definition of affine structures below: 
\begin{defi}
Let $X$ be a smooth manifold of dimension $d\geq 1$. An affine structure on $X$ is a distinguished atlas $\{(U_i,\alpha_i),\alpha_i:U_i\rightarrow \mathbb R^d\}$ whose transition functions are affine, i.e. for all $i,j$, $\alpha_i\circ\alpha_j^{-1}\in \GL_d(\mathbb R)\rtimes \mathbb R^d$.
\end{defi}
\begin{defi}
Let $X$ be a smooth manifold of dimension $d\geq 1$. An affine structure on $X$ is a flat torsion free connection $\nabla$ on the tangent bundle of $X$. 
\end{defi}
Given a flat torsion free connection, the exponential map can be used to define coordinates for a distinguished atlas. Conversely, given a distinguished atlas, a flat torsion free connection can be constructed by pulling back the standard connection on $\R^d$ to $U_i$ by $\alpha_i$. 

The affine structure defines a notion of geodesics and parallel transport of tensor fields on $X$. However, as $\partial$ is not necessarily the Levi-Civita connection of a Riemnannian metric there is no notion of distance or (globally) unit speed geodesics (see Example~\ref{ex:NonComplete} below). Parallel transport on $(X,\nabla)$ is locally constant, i.e. the parallel transport along a curve only depend on the homotopy type of the curve. The global holonomy of the connection is related to the linear part of the transition functions in the distinguished atlas. Essentially, if $H$ is a subgroup of $\GL_d(\R)$ then there is a distinguished atlas such that $\alpha_i\circ\alpha_j\in H\rtimes \mathbb R^d$ if and only if the holonomy of $\nabla$ lies in $H$. If the holonomy of $\nabla$ lies in $\SL_d(\R)$, then $(X,\nabla)$ is said to be \emph{special affine}. If the holonomy of $\nabla$ lies in $\SL_d(\Z)$, then $(X,\nabla)$ is said to be \emph{integral affine}. The holonomy of $\nabla$ lies in $\On_d(\R)$ if and only if $\nabla$ is the Levi-Civita connection of a Riemannian metric.

Any special affine manifold admits a parallel volume form, unique up to scaling. We will use $\mu_X$ to denote the unique parallel volume form of total volume 1. 

\begin{exam}\label{ex:AffineQuotient}
    Let $G$ be a subgroup of $\GL_d(\mathbb R)\rtimes \R^d$  and $\Omega\subset \R^d$ be a $G$-invariant open convex set with non-empty interior such that $X=\Omega/G$ is a smooth manifold. Then $(X,\nabla)$ where $\nabla$ is the connection inherited from $\mathbb R^d$ is an affine manifold. If in addition, $G\subset \SL_d(\mathbb R)\otimes \R^d$ or $G\subset \SL_d(\mathbb Z)\otimes \R^d$, then $(X,\nabla)$ is special affine or integral affine, respectively. Whenever $G$ is not commutative (see for example Example~\ref{ex:Hyperbolic} below), $X$ is not diffeomorphic to a torus. 
\end{exam}
\begin{exam}\label{ex:IsometricQuotient}
    As a special case of Example~\ref{ex:AffineQuotient}, let $G\subset O_d(\mathbb R)$. Then $X$ inherits the flat Euclidean metric on $\mathbb R^d$ and $\nabla$ is the Levi-Civita connection of this metric. When $G$ is cocompact, this case is categorized by Bieberbach's Theorem, saying that $(X,\nabla)$ is a finite quotient of a flat torus. 
\end{exam}
\begin{exam}\label{ex:NonComplete}
    Also a special case of Example~\ref{ex:AffineQuotient}, but contrasting Example~\ref{ex:IsometricQuotient}, let $\Omega=\mathbb R_+$ be the positive real numbers and $G$ be the group generated by $2\in \GL_1(\mathbb R)$, i.e. the map that sends $x\in \mathbb R_+$ to $2x$. The geodesics on $(X,\nabla)$ are exactly the images of geodesics on $\mathbb R_+$. In particular, starting at $1\in R_+$, moving in the direction of $-\partial/\partial x$, we get a geodesic $\gamma$ on $(X,\nabla)$ which is only defined for $t\in [0,1)$. This means that, although $X$ is compact, $(X,\nabla)$ is not geodesically complete. The holonomy of the generator of $\pi_1(X)$ defined by $\gamma|_{[0,1/2]}$ is $2\in \GL_1(\R)$, and an intuitive picture of what is happening is that each time the geodesics traverses $X$ its speed doubles (measured as the length of its tangent vector in a fixed frame near a point in $X$). Consequently, the speed diverges as $t\rightarrow 1$.  Note that $(X,\nabla)$ is not special. A well known conjecture by Markus states that a compact affine manifold is geodesically complete if and only if it is special. 
\end{exam}
\begin{exam}\label{ex:Hyperbolic}
    Let $\Omega$ be the cone in $\R^2\times \R$ defined by
    $$\Omega = \{(x_1,x_2,t)\in \R^2\times \R: Q(x_1,x_2,t):=x_1^2+x_2^2-t^2<0\}$$
    and consider the action on $\Omega$ generated by 
    $$\left(\begin{matrix} 
    2 & 0 & 0 \\ 
    0 & 2 & 0 \\
    0 & 0 & 2 
    \end{matrix}\right), \left(\begin{matrix} 
    \cosh(1) & 0 & \sinh(1) \\ 
    0 & 1 & 0 \\
    \sinh(1) & 0 & \cosh(1) 
    \end{matrix}\right) \text{ and } \left(\begin{matrix} 
    0 & -1 & 0 \\ 
    1 & 0 & 0 \\
    0 & 0 & 1 
    \end{matrix}\right).$$
    The latter two matrices preserve the level set of $Q$ and $X=\Omega/G$ is a compact affine manifold. 
\end{exam}

The topic of compact affine manifolds has been a very active research area at least since the 1970's (see for example \cite{Goldman}). However, in the context of the special Lagrangian torus fibrations in the SYZ-conjecture, the base spaces are expected to have singularities. More precisely, (in the generic case) the base is expected to be a sphere equipped with an integral affine structure which extends to the complement of a co-dimension 2 set. We will now look closer at a family of such objects. 

\subsection{Singular integral affine structures on reflexive polytopes}\label{sec:Polytopes}
Let $M_\R$ be a real vector spaces of dimension $d+1$, $N_\R=(M_\R)^*$, $M\subset M_\R$ a lattice and $N\subset N_\R$ the dual lattice. A reflexive polytope $\Delta\subset M_\R$ is a lattice polytope (i.e. the convex hull of a set of lattice points) such that its dual 
$$ \Delta^\vee = \{ n\in N_\R: \langle m,n \rangle \leq 1 \text{ for all } x\in \Delta\}\subset N_\R$$
is also a lattice polytope. For each facet $\sigma$ of $\Delta$ we will let $n_\sigma$ denote the corresponding vertex of $\Delta^\vee$ and for each facet $\tau$ of $\Delta^\vee$ we will let $m_\tau$ denote the corresponding vertex of $\Delta$. We will use $\sigma^\circ$ and $\tau^\circ$ to denote the relative interior of $\sigma$ and $\tau$, respectively and $\Star(m)$, for $m\in \partial \Delta$ to denote the union of all closed faces in $\partial \Delta$ containing $m$. 
\begin{exam}\label{ex:Polytope}
    Note that $\Delta$ is homeomorphic to the unit sphere. We will now explain how to a define a natural distinguished atlas on a large subsets of the boundary $\partial\Delta$ of $\Delta$. First of all, for each facet $\tau$ of $\Delta^\vee$, pick an open set $U_\tau$ in $\Star(m_\tau)$ such that the sets in the collection $\{U_\tau\}_\tau$ are pairwise disjoint and $\cup U_\tau$ covers a large part (say everything except for a (d-2)-dimensional subset) of the $d-1$ skeleton in $\partial\Delta$. The distinguished atlas will consist of one coordinate chart on $U_{\tau}$ for each facet $\tau$ of $\Delta^\vee$ and one coordinate chart on $\sigma^\circ$ for each pair of facets $\sigma$ and $\tau$ of $\Delta$ and $\Delta^\vee$ such that $m_\tau\in \sigma$ or, equivalently, $n_\sigma\in \tau$. 

    For each facet $\tau$, let $n_1,\ldots,n_d$ be a set of generators of the lattice $N\cap m_\tau^\perp$ in the orthogonal complement of $m_\tau$ and let $\alpha_{\tau}:\Star(m_\tau)\rightarrow \R^d$ be defined by
    $$ \alpha_\tau(m) = \left(\langle m,n_1 \rangle, \ldots, \langle m,n_d \rangle\right). $$
    We define the distinguished atlas as the set of coordinate charts $\{(U_\tau,\alpha_\tau|_{U_\tau})\}_\tau$ together with the coordinate charts $\{(\sigma^\circ,\alpha_\tau|_{\sigma^\circ})\}_{\sigma,\tau}$ where the latter family is indexed by all pairs of facets $\sigma$ and $\tau$ of $\Delta$ and $\Delta^\vee$, respectively, such that $m_\tau\in \sigma$. 
\end{exam}
\begin{lem}
    The coordinate charts in Example~\ref{ex:Polytope} define an integral affine structure. 
\end{lem}
\begin{proof}
    (See also \cite[Lemma~3]{AH})
    It suffice to show that if $n_\sigma\in \tau\cap \tau'$, then $\alpha_\tau|_{\sigma^\circ} \circ (\alpha_{\tau'}|_{\sigma^\circ})^{-1}\in \SL_d(\Z)\rtimes \R^d$. First of all, let $n_1,\ldots,n_d$ be the generators for $N\cap m_\tau^\perp$ used in the definition of $\alpha_\tau$. Since $\langle m_\tau,n_\sigma \rangle = 1 \not= 0$ we get that $n_\sigma\notin m_\tau^\perp$, hence if $\langle m,n_\sigma \rangle=0$ and $\langle m,n_i \rangle=0$ for all $i$, then $m=0$. It follows that $\alpha_\tau$ define a non-degenerate affine map from the affine subspace spanned by $\sigma$ to $\mathbb R^d$. The same argument can be made for $\alpha_{\tau'}$ Consequently, $\alpha_\sigma \circ \alpha_{\sigma'}$ is affine. 
    
    To prove the lemma it thus suffices to verify that its linear part is in $\SL_d(\Z)$. To do this, we will show that both $\alpha_\sigma$ and $\alpha_{\sigma'}$ maps $M\cap \affspan(\sigma)$ to $\Z^d\in \R^d$. First of all, since $n_1,\ldots,n_d\in N$, we get that $\alpha_\tau(m)\in \Z^d$ if $m\in M$. For the converse, assume $m\in M\cap \affspan(\sigma)$ and $\alpha_\tau(m)\in \Z^d$. Let $e_0,\ldots,e_d$ be generators of $M$. We have
    $$ n_\sigma\wedge n_1 \wedge  \ldots \wedge n_d = \langle m_\tau,n_\sigma \rangle e_0\wedge \ldots \wedge e_d = e_0\wedge \ldots \wedge e_d, $$
    hence $n_\sigma,n_1,\ldots,n_d$ generate $N$. By assumption $\langle m,n_\sigma \rangle = 1\in \Z$ and $\langle m,n_i \rangle\in \Z$ for each $i$. Consequently, $m\in M$. The corresponding statement for $\alpha_{\tau'}$ is proved in the same way. 
\end{proof}
\begin{rmk}
    The definition of the integral affine structure in Example~\ref{ex:Polytope} involves a choice of sets $\{U_\tau\}_\tau$. However, note that this choice only influence the integral affine structure on the $(d-1)$-skeleton of $\partial \Delta$. More precisely, the choice of sets $\{U_\tau\}_\tau$ is equivalent to picking a ''boundary'' in each $(d-1)$-face of $\Delta$ which divides the face into a number of open regions, each one of these regions corresponding to one of the vertices of the face. 
    In particular, the parallel volume form of the integral affine structures defines a probability measure $\mu_M$ on $\partial\Delta$ independent of this choice. Similarly, we get a probability measure $\nu_N$ on $\partial\Delta^\vee$. 
\end{rmk} 
\begin{rmk}\label{rem:ChartsFromPhi}
    Assume $\Phi$ is a convex function on $M_\R$ such that $|\sup_{n\in \Delta^\vee} \langle m,n \rangle-\Phi|$ is bounded. Then, as argued in \cite{AH}, $\Phi$ dictates a natural choice for $\{U_\tau\}_\tau$, namely 
    $$ U_\tau = \Delta\setminus \partial^c\Phi^c(\Delta^\vee\setminus \tau^\circ) $$
    where $\Phi^c$ is the $c$-transform of $\Phi$ and $\partial^c\Phi^c$ its $c$-gradient with respect to the cost function $c(m,n) = -\langle m,n \rangle_{M,N} $ and $\langle \cdot,\cdot \rangle_{M,N}$ is the pairing of $M_\R$ and $N_\R$. 
    In sufficiently regular cases, this is equivalent to putting $U_\tau = \partial^c\Phi^c(\tau^\circ)$
\end{rmk}
\begin{rmk}
    The integral affine structure in Example~\ref{ex:Polytope} relates directly to a class of families of Calabi-Yau manifolds relevant to SYZ mirror symmetry. When $\Delta$ is Delzant, it defines a toric Fano manifold $Z_\Delta$ of complex dimension $d+1$. Let $s_0$ be the $(\C^*)^{d+1}$ invariant section of its anticanonical bundle and $s$ be a generic section of its anticanonical bundle and $Y$ be the family of Calabi-Yau manifolds given by
    \begin{equation} \label{eq:Family} Y = \{ts+s_0:t\in \C^*\}\subset Y. \end{equation}
    Then $\Delta^\vee$ arise as the bounded component of a limiting amoeba of $Y$ and the coordinate charts on $\Delta^\vee$ described by the procedure above correspond to a set of holomorphic coordinate charts on $Y$ which cover a large part of the Calabi-Yau manifolds near the central fiber (see \cite{LiFermat}, and also \cite{HZ,GS} for more details). Moreover, solvability of a certain Monge-Ampère equation on these has direct implications on the SYZ conjecture (see Corollary~\ref{cor:SYZ} below).
\end{rmk}

\subsection{Hessian Metrics}
\begin{defi}
Let $(X,\nabla)$ be an affine manifold. A Hessian metric on $(X,\nabla)$ is a Riemannian metric $g$ on $X$ which is locally the Hessian of a convex function, i.e. there is a collection of open sets $\{U_i\}$ covering $X$ and convex functions $\{\phi_i:U_i\rightarrow \R\}$ such that  
$$ g|_{U_i}=\nabla d\phi_i = \sum_{jk} \frac{\partial^2\phi}{\partial x_j\partial x_k}dx_j\otimes dx_k $$
for each $i$, where $x_1,\ldots,x_d$ are coordinates on $U_i$ from the distinguished atlas. 
\end{defi}
We will say that a collection of convex functions $\{\phi_i:U_i\rightarrow \R\}$ (not necessarily smooth) such that $\phi_i-\phi_j$ is affine, i.e. $\nabla d(\phi_i-\phi_j)=0$, for all $i$ is a \emph{weak} Hessian metric. If $(X,\nabla,g)$ is a special Hessian manifold and $u$ is a function on $X$ such that $\phi_i+u$ is convex for each $i$, then we define the Monge-Ampère measure $\MA(g+\nabla d\phi)$ as the Monge-Ampère measure of $\phi_i+u$ in the sense of Aleksandrov.

\begin{exam}
Let $\Omega$ and $G$ be as in Example~\ref{ex:AffineQuotient} and $\Phi$ a smooth strictly convex function on $\Omega$ such that $\Phi\circ g-\Phi$ is affine for all $g\in G$. Then the Riemannian metric on $\Omega$ defined by $\Phi_{ij}dx_i\otimes dx_j$ is $G$ invariant and its pushforward under the quotient map defines Hessian metric on $X$. In particular, if $\Omega=\mathbb R$, and $G=\mathbb Z$ acting on $\mathbb R$ by translations, we may pick $\Phi=x^2/2$, since $\Phi(x-m)-\Phi(x) = -mx+m^2/2$ is affine in $x$ for each $m\in \mathbb Z$. In Example~\ref{ex:NonComplete} we can pick $\Phi(x) = -\log(x)$, since $\log(2^mx)-\log(x)=m\log(2)$ is affine in $x$. In Example~\ref{ex:Hyperbolic} we can pick $\Phi=-\log(-Q)$.
\end{exam}

\begin{exam}\label{ex:PolytopeHessian}
    In Example~\ref{ex:Polytope}, any $\Phi$ defines a weak Hessian metric. On the charts given by $(\sigma^\circ,\alpha_\tau)$ we use $\Phi|_{\sigma^\circ}$ and on the charts given by $(U_\tau,\alpha_\tau)$ we use $(\Phi-n)|_{\sigma^\circ}$ for any vertex $n$ of $\tau$.
\end{exam}

\begin{thm}\label{thm:CompactHessianIsConvex}
Let $(X,\nabla,g)$ be a compact Hessian manifold. Then $X$ is a quotient $\Omega/G$, for some convex $\Omega\subset \R^d$ and $G$ as in Example~\ref{ex:AffineQuotient}. Moreover, $g=\nabla d\Phi$ for some proper function $\Phi:\Omega\rightarrow \R$.
\end{thm}
The proof of Theorem~\ref{thm:CompactHessianIsConvex} is mostly elementary, but it helps to introduce some terminology: Let $p:\Omega\rightarrow X$ be the universal covering space of $X$. We may pull back the affine structure and the Hessian metric on $X$ to get a Hessian structure on $\Omega$. The fact that $\Omega$ is simply connected implies the following:
\begin{itemize}
    \item The pullback of $g$ can be written as $\nabla d\Phi$ for some globally defined $\Phi:X\rightarrow \R$.
    \item $\Omega$ admits $d$ linearly independent parallel 1-forms.
\end{itemize}
Fix $x_0\in \Omega$.
The first of these points can be showed by defining $\Phi(x) = \int_{\gamma_{x_0,x}} \int_{\gamma_{x_0,x}} g$, where $\gamma_{x_0,x}$ is some curve in $\Omega$ from $x_0$ to $x$ and $\int_{\gamma_{x_0,x}} g$ is the 1-form given by integrating $g$ along $\gamma_{x_0,x}$. The second point can be showed by fixing a basis of the co-tangent space of $\Omega$ at $x_0$ and extending this by paralell transport. Both of these procedures only depend on the homotopy type of the $\gamma_{x_0,x}$. Since $\Omega$ is simply connected they are thus independent of $\gamma_{x_0,x}$. 

Let $m_1,\ldots,m_d$ be linearly independent parallel 1-forms on $\Omega$ and, as above, $x_0\in \Omega$. The \emph{developing map} at $x_0$ defined by $m_1,\ldots,m_d$ is the map $dev:\Omega\rightarrow \R$ given by
$$ dev(x) = \left(\int_{\gamma_{x_0,x}} m_1,\ldots,\int_{\gamma_{x_0,x}} m_d\right) $$
where as above $\gamma_{x_0,x}$ is any curve in $\Omega$ from $x_0$ to $x$. 

We will say that a geodesic segment $\gamma:[0,1]\rightarrow X$ (with respect to $\nabla$) is maximal if it can't be continued beyond time 1, i.e. there is no other geodesic $\gamma':[0,1+\epsilon]\rightarrow X$, $\epsilon>0$ such that $\gamma=\gamma|_{[0,1]}$. The following is an important property of compact Hessian manifolds:
\begin{lem}\label{lem:PhiInft}
    Assume $\gamma:[0,1]\rightarrow \Omega$ is a maximal geodesic segment. Then $\frac{d}{dt}\Phi(\gamma(t))\rightarrow +\infty$ as $t\rightarrow 1$. 
\end{lem}
\begin{proof}
    Let $f(t)= \Phi(\gamma(t))$. We have 
    \begin{eqnarray*}
        \frac{d}{dt}\Phi(\gamma(t)) & = & f'(t) \\
        & = & f'(0)+\int_0^t f''(s)ds  \\
        & \geq & f'(0)+\int_0^t (f''(s))^{1/2} ds \\
        & = & |\gamma|_g \\
        & = & +\infty
    \end{eqnarray*}
    where $|\gamma|_g$ denotes the length of $\gamma$ with respect to $g$. To see that this is infinite, note that since $(X,g)$ is a compact Riemannian manifold, $|\gamma|_g$ is either infinite, or $\gamma(t)$ has a limit $x$ as $t\rightarrow 1$. But in the latter case, we can pick affine coordinates around $x$ and continue $\gamma$ as a stright line in these coordinates. This contradict the maximality of $\gamma$.
\end{proof}
\begin{proof}[Proof of Theorem~\ref{thm:CompactHessianIsConvex}, following \cite{ShimaYagi}]
The theorem is proved by showing that the developing map is injective and that its image is convex. To do this, it suffices to prove that if $x_1, x_2, x_3 \in \Omega$, and there are geodesics $\gamma_{x_1,x_2}$ and $\gamma_{x_2,x_3}$ connecting $x_1$ with $x_2$ and $x_2$ with $x_3$, respectively, then there is a geodesics connecting $x_1$ with $x_3$. Since every pair of points can be connected by a piecewise geodesic segment the theorem then follows by repeated applications of this statement. More preciesly, if $\gamma:[0,1]\rightarrow \Omega$ is a piece geodesic segment which is geodesic on $[t_1,t_2]$ and $[t_2,t_3]$ then the statement above can be used to modify $\gamma$ on $[t_1,t_3]$ to a piecewise geodesic which is geodesic on $[t_1,t_3]$. 

Assume for a contradiction that there is no geodesic from $x_1$ to $x_3$. Without loss of generality we may assume that there is a geodesic $\gamma_t$ from $\gamma_{x_1,x_2}(1-t)$ to $\gamma_{x_2,x_3}(t)$ for all $t<1$. Moreover, by considering the geodesic emanating from $x_1$ in the direction given by $\dot\gamma_{x_2,x_3}-\dot\gamma_{x_1,x_2}$ and the geodesic emanating from $x_3$ in the direction $\dot\gamma_{x_1,x_2}-\dot\gamma_{x_2,x_3}$ we get two maximal geodesic segments $\gamma_1:[0,a)\rightarrow \mathbb R$, $\gamma_2:[0,b)\rightarrow \mathbb R$ where $a+b\leq 1$. By Lemma~\ref{lem:PhiInft}, 
$$\langle d\Phi, \dot\gamma_1\rangle  = \langle d\Phi, \dot\gamma_{x_2,x_3}-\dot\gamma_{x_1,x_2}\rangle >>0$$ 
close to $\gamma_1(a-\epsilon)$ in $\Omega$ and 
$$\langle d\Phi \dot\gamma_1 \rangle = \langle d\Phi, \dot\gamma_{x_1,x_2}-\dot\gamma_{x_2,x_3} \rangle >>0$$ close to $\gamma_2(b-\epsilon)$ in $\Omega$. This contradicts convexity of $\Phi$. 
\end{proof}

\begin{rmk}
An alternative way of thinking about Theorem~\ref{thm:CompactHessianIsConvex} is that if $(X,\nabla,g)$ is a compact Hessian manifold, then any curve on $X$ can be continuously deformed to a geodesic (with respect to $\nabla$) while keeping its endpoints fixed. 
\end{rmk}

\section{Duality of Hessian Manifolds}\label{sec:DualHessianStructures}
Let $U$ be a convex subset of $\R^d$, $\Phi:U\rightarrow \R$ be smooth and strictly convex and put $g := \sum_{ij} \frac{\partial^2\Phi}{\partial x_i\partial x_j}dx_i\otimes dx_j$. Let $\Psi:\partial\Phi(U)\rightarrow \R$ be the Legendre transform of $\Phi$ and $y_1,\ldots,y_d$ be the (non-affine) coordinates on $U$ defined by the diffeomorphism $\partial\Phi$, i.e. $(y_1,\ldots,y_d)=\partial\Phi(x_1,\ldots,x_d)$. Let $\Phi_{ij} = \frac{\partial^2\Phi}{\partial x_i\partial x_j}$, $\Psi_{ij} = \frac{\partial^2\Psi}{\partial y_i\partial y_j}$ and $(\Phi^{ij})$ denote the inverse of the matrix $(\Phi_{ij})$. Using \eqref{eq:HessLeg} we get
\begin{eqnarray}
    g & = & \sum_{ij} \Phi_{ij}dx_i\otimes dx_j \nonumber \\
    & = & \sum_{ij} \Phi^{ij}\circ \partial \Phi^* dy_i\otimes dy_j \nonumber \\
    & = & \sum_{ij} \Psi_{ij}\circ \partial \Phi dy_i\otimes dy_j.
\end{eqnarray}
This suggest there is another affine structure on $U$, different from the one inherited from $\R^d$, and for which $g$ is still a Hessian metric. More precisely, if $\nabla'$ is the affine structure on $U$ defined by pulling back the affine structure on $(\R^d)^*$ by $\partial\Phi$, then $(U,\nabla',g)$ is a Hessian metric and a potential for $g$ is given by $\Psi\circ \partial\Phi$. 

This operation can be performed locally on a Hessian manifold $(X,\nabla,g)$ to define a \emph{dual Hessian structure} $(X,\nabla',g)$. 
\begin{lem}\label{lem:DualConnection}
Let $(X,\nabla,g)$ be a Hessian manifold and assume $\{U_i,\alpha_i\}$ is a set of affine coordinate charts covering $X$ such that for each $i$ $\alpha_i(U_i)\subset \R^d$ is convex and there is $\phi_i:U\rightarrow \R$ such that $g|_{U_i}=\nabla d\phi_i$. Then
\begin{enumerate}
    \item The affine structure $\nabla_i$ on $U_i$ defined by pulling back the affine structure on $(\R^d)^*$ by $\partial(\phi_i\circ \alpha^{-1})$ has the following form: $\nabla'_i = 2\nabla_g-\nabla$, where $\nabla_g$ is the Levi-Civita connection of $g$. Consequently, the locally defined connections $\{\nabla'_i\}$ glue together to form a global connection $\nabla'$ on $X$.  
    \item The object $(X,\nabla',g)$ is a Hessian manifold and a set of potentials of $g$ is given by $\{\phi^*_i\circ \partial\phi:U_i\rightarrow \R\}$.
\end{enumerate} 
\end{lem}
We will refer to \cite{ShimaYagi} or \cite{ShimaBook} for a proof of the first bullet point above. Assuming the first bullet point, the second bullet point then follows from the discussion preceding Lemma~\ref{lem:DualConnection}. 

The Hessian manifold $(X,\nabla',g)$ is called the dual of $(X,\nabla,g)$. 


\subsection{Duals of compact Hessian manifolds}\label{sec:DualCptCase}
Let $(X,\nabla,g)$ be a compact Hessian manifold. By Theorem~\ref{thm:CompactHessianIsConvex} we can assume $(X,\nabla,g)$ is of the form $\Omega/G$ for some convex $\Omega$ and $G$ as in Example~\ref{ex:AffineQuotient}, equipped with a smooth, strictly convex function $\Phi:\Omega\rightarrow \R$ whose Hessian is the pullback of $g$. 
Note that $\partial\Phi$ defines a diffeomorphism from $\Omega$ to $\partial\Phi(\Omega)$. We can define an action of $G$ on $\partial \Phi(\Omega)\subset (\R^d)^*$ by pushing forward the action of $G$ on $\Omega$, in other words:
$g(\partial\Phi(m)) = \partial\Phi(m')$ if $g(m)=m'$.  
By construction, $\partial\Phi$ defines an equivivariant diffeomorphism from $\Omega$ to $\partial\Phi(\Omega)$, hence $\Phi(\Omega)/G$ is a compact smooth manifold. We can equip this with the affine structure inherited from $(\R^d)^*$. 
\begin{lem}\label{lem:QuotientIsDual}
$\partial\Phi(\Omega)/G$ is affinely isomorphic to $(X,\nabla')$.
\end{lem}
\begin{proof}
    Let $\{U_i\}$ be a set of open sets in $\Omega$ such that the quotient map $p:\Omega\rightarrow X$ restricts to a homeomorphism on $U_i$ for each $i$. For each $i$, we may pick $\phi_i:=\Phi|_{U_i}\circ (p|_{U_i})^{-1}$ as a potential for $g$ on $p(U_i)$. As noted above, $\partial\Phi$ defines a diffeomorphism from $X$ to $\partial\Phi(\Omega)/G$. Restricting to $U_i$, this is exactly the map defined by $\partial\phi_i$. Consequently, the pullback of the connection on $\partial\Phi(\Omega)/G$ is exactly $\nabla'$. 
\end{proof}

We will say that two Hessian metrics $g$ and $g'$ on $(X,\nabla)$ lies in the same \emph{class} if $g-g'=\nabla d\phi$ for some smooth function $\phi:X\rightarrow \R$. It turns out that both the set $\partial\Phi(\Omega)$ and the action above only depend on the class of $g$. 

\begin{thm}\label{thm:DualQuotient}
    Assume $(X,\nabla,g)$ is a compact Hessian manifold and $g'$ is a Hessian metric on $(X,\nabla)$ in the same class as $g$. Let $\Phi:\Omega\rightarrow \R$ and $\Phi':\Omega\rightarrow \R$ be potentials of $g$ and $g'$, respectively. Then $\partial\Phi(\Omega)=\partial\Phi'(\Omega)$ and the action of $G$ induced on $\partial\Phi(\Omega)$ by $\Phi$ is the same as the action induced by $\Phi'$. 
\end{thm}
\begin{proof}
    By properness of $\Phi$ and convexity of $\Omega$ we get that $\Phi$ extends to a lsc convex function on $\R^d$ by putting $\Phi(x)=+\infty$ if $x\notin \Omega$. It follows that $\overline{\partial\Phi(\Omega)} = \overline{\{\Phi^*<+\infty \}}$. Similarly, $\overline{\partial\Phi'(\Omega)} = \overline{\{(\Phi')^*<+\infty \}}$. However, since $\Phi-\Phi'$ is bounded (it descends to a smooth function on $X$), $\overline{\{\Phi^*<+\infty \}}=\overline{\{(\Phi')^*<+\infty \}}$. Since two convex open sets are the same if their closures are the same, we get that $\partial \Phi(\Omega)=\partial \Phi'(\Omega)$. 

    Assume $m\in \Omega$, $g\in G$, $n_0=\partial\Phi(m)$ and $n_1=\partial\Phi(gm)$. In other words, with respect to the action induced by $\Phi$, $g$ maps $n_0$ to $n_1$. Let $m_0'\in \Omega$ satisfy $n_0=\partial \Phi'(m_0')$. We need to show that $n_1=\partial \Phi'(gm_0')$, and hence that with respect to the the action induced by $\Phi'$, $g$ maps $n_0$ to $n_1$.

    To do this, let $\gamma$ be a curve from $m_0$ to $gm_0$, $\gamma_0$ be a curve from $m_0'$ to $m_0$ and note that $g\circ\gamma_0$ is a curve from $gm_0'$ to $gm_0$. Let $\gamma'$ be the curve from $m_0'$ to $gm_0'$ given by first traversing $\gamma_0$, then traversing $\gamma$ and then traversing $g\circ \gamma_0$ backwards. The fact that $\sum_{ij}\Phi_{ij}dx_i\otimes dx_j$ is invariant under $G$ means 
    $$ \int_{\gamma_0} \sum_{ij}\Phi_{ij}dx_i\otimes dx_j = \int_{g\circ \gamma_0} \sum_{ij}\Phi_{ij}dx_i\otimes dx_j $$
    and the fact that $\Phi-\Phi'$ is invariant under $G$ means that $\int_\gamma \sum_{ij}(\Phi_{ij}-\Phi_{ij}')dx_i\otimes dx_j = 0$. We have 
    \begin{eqnarray*}
        \partial \Phi'(gm_0') & = & \partial \Phi'(m_0')+\int_{\gamma'} \Phi'_{ij} dx_i\otimes dx_j \\  
        & = & n_0 + \int_{\gamma_0} \Phi'_{ij} dx_i\otimes dx_j + \int_{\gamma} \Phi'_{ij} dx_idx_j - \int_{g\circ \gamma_0} \Phi'_{ij} dx_i\otimes dx_j \\ 
        & = & \partial \Phi(m_0)+\int_\gamma \Phi_{ij}dx_i\otimes dx_j \\
        & = & \partial \Phi(gm_0) \\
        & = & n_1.
    \end{eqnarray*}
\end{proof}

\begin{cor}\label{cor:DualCptCase}
Let $(X,\nabla,g)$ be a compact Hessian manifold and $(X,\nabla',g)$ its dual. Up to affine isomorphism, $\nabla'$  only depends on the class of $g$. In other words, if $g'$ is another Hessian metric on $(X,\nabla)$ such that $g-g'=\nabla d\phi$ for a smooth function $\phi$ on $X$, then the dual of $(X,\nabla,g')$ is $(X,\nabla',g')$. 
\end{cor}
\begin{proof}
    Follows directly from Lemma~\ref{lem:QuotientIsDual} and Theorem~\ref{thm:DualQuotient}.
\end{proof}
\begin{rmk}
    In analogy with the complex geometric case, if $(X,\nabla,g)$ is a compact Hessian manifold, then it is possible to construct an 'affine line bundle' over $X$ such that the collection of local potentials of any Hessian metric in the same class as $g$ can be interpreted as a metric on the affine line bundle. See \cite{HO19} for details. 
\end{rmk}

\subsection{Duals of reflexive polytopes}\label{sec:DualPolytope}
In Section~\ref{sec:Polytopes} we described natural integral affine structures on large subsets of the boundary of reflexive polytopes. We will now describe their duals. The result we will present relies on some subtle regularity properties of the Hessian metric, hence might not be directly applicable, but arguably it gives a good intuitive picture. 

First of all, we fix a reflexive polytope $\Delta\subset M_\R$ and a convex function $\Phi$ on $M_\R$ such that $|\sup_{n\in \Delta^\vee} \langle m,n \rangle -\Phi(m)|$ is bounded. As in Remark~\ref{rem:ChartsFromPhi}, we will let $U_\tau = \Delta\setminus \partial^c\Phi^{-1}(\Delta^\vee\setminus \tau^\circ)$.

Define $\Phi^c:N_\R\rightarrow \R$ by 
\begin{equation} \label{eq:cTransformInDual} \Psi(n) = \Phi^c(n) = \sup_{m\in \partial\Delta} \langle m,n \rangle - \Phi(m). \end{equation}
Note that $\Delta^\vee$ is also a reflexive polytope and $(\Delta^
\vee)^\vee = \Delta$, hence $|\sup_{m\in (\Delta^\vee)^\vee} \langle m,n \rangle -\Psi(n)|$ is bounded. 
Following Remark~\ref{rem:ChartsFromPhi} again, we set $U_\sigma = \Delta^\vee\setminus \partial^c\Psi^{-1}(\Delta\setminus \sigma^\circ)$. 
Together with the recipe in Example~\ref{ex:Polytope}, this determines integral affine structures on large subsets of $\partial\Delta$ and $\partial\Delta^\vee$. 

We will assume that  the singular Hessian metrics induced by $\Phi$ and $\Psi$ are smooth. In addition, we will assume that $\partial^c\Phi:\Delta\rightarrow \Delta^\vee$ is a homeomorphism. The latter property has some very useful consequences:
\begin{itemize}
\item $\partial^c\Psi = (\partial^c\Phi)^{-1}$ is also a homeomorphism.
\item The definition of $U_\tau$ and $U_\sigma$ in Remark~\ref{rem:ChartsFromPhi} reduces to $U_\tau = \partial^c\Psi(\tau)$ and $U_\sigma = \partial^c\Phi(\sigma)$.
\item $\partial^c\Phi(\sigma)\subset \Star(n_\sigma)$ and $\partial^c\Psi(\tau)\subset \Star(m_\tau)$ for all facets $\sigma$ of $\Delta$ and $\tau$ of $\Delta^\vee$ (see \cite[Lemma~8]{AH}).
\item Let $m\in \sigma^\circ$. Then the generators of $N\cap m_\tau^\perp$ and $M\cap n_\sigma^\perp$ in the definitons of $\alpha_\tau$ and $\alpha_\sigma$ can be picked so that 
%
$$ \alpha_\sigma\circ \partial^c\Phi = \partial(\Phi\circ \alpha_\tau^{-1})\circ \alpha_\tau $$
and if $m\in U_\tau$ then the generators can be picked so that
%
$$ \alpha_\sigma\circ \partial^c\Phi = \partial((\Phi-n_\sigma)\circ \alpha_\tau^{-1})\circ \alpha_\tau $$ 
(see \cite[Corollary~3 and Lemma~14]{AH}). 
\end{itemize}
\begin{thm}\label{thm:DualPolytope}
Let $(X,\nabla,g)$ be the integral affine structure on $\partial\Delta$ induced by $\Phi$ and $(X,\nabla',g)$ its dual. Then, under the assumptions above, $(X,\nabla)$ is affinely isomorphic to the integral affine structure on $\Delta^\vee$ induced by $\Psi$ and $g$ coincides with the Hessian metric defined by $\Psi$. 
\end{thm}
\begin{proof}
    We may use the homeomorphism $\partial^c\Phi$ to identify the domain of the affine structure on $\partial\Delta$ with the domain of the affine structure of $\partial\Delta^\vee$. Moreover, by the last bullet point preceding this theorem, locally $\partial^c\Phi$ is the classical gradient of the potential of the Hessian metric, hence the theorem follows in the same way as Theorem~\ref{thm:DualQuotient}.
\end{proof}
\begin{rmk}
    The assumption that $\partial^c\Phi$ defines a homeomorphism is rather subtle. There does not seem to be any obvious ways to write down explicit examples $\Phi$ satisfying this. On the other hand, the second and third bullet point, and hence also the fourth, can be showed directly for the solution to important Monge-Ampère equations when $\Delta$ is the standard unit simplex or the unit cube (\cite[Lemma~16 and Lemma~17]{AH}).
\end{rmk}

\section{Pairings and Optimal Transport Problems}\label{sec:Pairings}
The purpose of the detailed discussion about dual Hessian structures in Section~\ref{sec:DualCptCase} and Section~\ref{sec:DualPolytope} are to provide heuristic optimal transport theoretic motivations for the functionals in \cite{HO19,HJMM,LiToric,AH}. In order to do this, the picture need to be completed with an extra structure in the form of a pairing between $(X,\nabla,g)$ and its dual. We will do this in the two cases discussed in Section~\ref{sec:DualCptCase} and Section~\ref{sec:DualPolytope}. 

\subsection{Compact Hessian Manifolds}
Assume $(X,\nabla,g)$ is a compact Hessian manifold. Pick $\Omega$ and $G$ as in Example~\ref{ex:AffineQuotient} so that $X=\Omega/G$ (c.f. Theorem~\ref{thm:CompactHessianIsConvex}) and let $\Phi:\Omega\rightarrow \R$ be a potential of the pullback of $g$. Identify $(X,\nabla')$ with $X^\vee:=\partial\Phi(\Omega)/G$ (c.f. Lemma~\ref{lem:QuotientIsDual}) and let $\Psi$ be a potential of $g$ ($\Psi$ is the Legendre transform of $\Phi$). We define the pairing $\langle \cdot,\cdot \rangle_\Phi:X\times X^\vee\rightarrow \R$ by
\begin{equation} \label{eq:PairingCompactCase} -\langle x,y \rangle_\Phi = \sup_{g\in G} \Phi(gx)+\Psi(y)-\langle gx,y \rangle = \sup_{g\in G} \Phi(x)+\Psi(gy)-\langle x,gy \rangle \end{equation}
where $\langle \cdot,\cdot \rangle$ is the standard pairing of $\R^d$ and $(\R^d)^*$ (note that the formula \eqref{eq:PairingCompactCase} describes a $G\times G$-invariant function on $\Omega\times \partial\Phi(x)$). Let $\mu$ be a probability measure on $X$, $\nu_X$ be the unique parallel probability measure on $X^\vee$ and consider the optimal transport problem defined by $(X,\mu), (X^\vee,\nu)$ and $c(x,y) = -\langle x,y\rangle_{\Phi,\Psi}$ we arrive at the dual problem 
\begin{equation} \label{eq:FuncCptCase} \int_X u \mu + \int_{X^\vee} u^c \nu \end{equation}
where 
$$ u^c(y) = \sup \langle x,y \rangle_\Phi - u(x) $$
which is the functional used in \cite{HO19} to produce solutions to the equation $\MA(g+\nabla du)=\mu$ on $X$.

\subsection{Reflexive Polytopes}
Assume $\Delta\subset M_\R$ is a reflexive polytope. Let $\mu_M$ be the parallel volume form on $\partial\Delta$ and $\nu$ a measure on $\partial\Delta^\vee$. Let $\langle \cdot,\cdot \rangle_{M,N}$ be the pairing of $M_\R$ and $N_\R$. Considering the optimal transport problem given by $(\partial\Delta_{reg},\mu_M)$, $(\partial\Delta_{reg},\nu)$ and $c(m,n) = -\langle m,n \rangle_{M,N}$, we arrive at the dual problem 
$$ \int_{\partial\Delta} \Phi \mu_M + \int_{\partial\Delta^\vee} \Phi^c \nu $$
or, equivalently,
$$ \int_{\partial\Delta} \Psi^c \mu_M + \int_{\partial\Delta^\vee} \Psi \nu $$
which is the functional used in \cite{HJMM, AH, LiToric}.

\section{Monge-Ampère Equations}\label{sec:MA}
The first general existence result for Monge-Ampère equations on special affine manifolds is due to Cheng and Yau:
\begin{thm}[\cite{ChengYau}, see also \cite{Del89}]\label{thm:CY}
    Let $(X,\nabla,g)$ be a compact special Hessian manifold and $\mu$ a volume form on $X$, then there is a unique smooth function $\phi$ on $X$ such that $g+\nabla d\phi$ is positive and 
    \begin{eqnarray} \det(g+\nabla d\phi) = \mu. \label{eq:CY} \end{eqnarray}
\end{thm}
\begin{rmk}
    Theorem~\ref{thm:CY} has some interesting consequences. Let $\mu=\mu_X$ and let $\phi$ be a solution to \eqref{eq:CY}. Pull back $g+\nabla d\phi$ to the covering space $\Omega\subset \R^d$ and let $\Phi$ be a global potential. By Theorem~\ref{thm:CompactHessianIsConvex}, $\Phi$ is a proper function which solves the equation 
    $$ \det(D^2\Phi) = c $$ on $\Omega$. It then follows from a theorem by Jurgen, Pogorelov and Calabi that $\Omega=\R^d$ and $\Phi$ is a quadratic polynomial. It follows that $g+\nabla d\phi$ is a flat metric, $\nabla$ is the Levi-Civita connection of $g+\nabla d\phi$ and $G$ acts on $\Omega=\R^d$ by rigid motions, hence $(X,\nabla)$ is of the type described by Example~\ref{ex:IsometricQuotient}. 
\end{rmk}

Using the variational approach described in Section~\ref{sec:Pairings} Theorem~\ref{thm:CY} was generalized by Önnheim and the author to:
\begin{thm}[\cite{HO19}]
    Let $(X,\nabla,g)$ be a compact special Hessian manifold and $\mu$ a probability measure on $X$, then there is a unique function $\phi$ on $X$ such that \eqref{eq:CY} holds.
\end{thm}

\begin{rmk}
    See also \cite{GuedjTo} for an approach to degenerate Monge-Ampère equation on compact Hessian manifolds based on the Perron method. 
\end{rmk} 

As before, let $\Delta$ be a reflexive polytope and $\Psi:N_\R$ satisfy $|\sup_{m\in \Delta} \langle m,n \rangle -\Psi|$ is bounded. We will use $\MA(\Psi)$ to denote the Monge-Ampère of the weak Hessian metric induced by $\Psi$ (c.f. Example~\ref{ex:PolytopeHessian})
\begin{thm}\cite{HJMM}\label{thm:symmetric}
Let $\Delta$ be the $(d+1)$-dimensional unit simplex
$$ \Delta = \conv\{-\sum e_i, (d+1)e_0,\ldots,(d+1)e_d\} $$
and assume $\nu$ is a probability measure on $\partial\Delta^\vee$ which is invariant under the symmetry group of $\Delta^\vee$, then there is a unique $\Psi:N_\R\rightarrow \R$ satisfying $|\sup_{m\in \Delta} \langle m,n \rangle -\Psi|$ such 
\begin{equation} \label{eq:MAPolytope} \MA(\Psi) = \nu. \end{equation}
\end{thm}
\begin{rmk}
    For $\nu=\nu_N$, the existence part of Theorem~\ref{thm:symmetric} was proved in \cite{LiFermat}, where $\Psi$ is constructed as a limit of the potentials of Calabi-Yau metrics on the associated family \ref{eq:Family}.
\end{rmk}

\begin{defi}
    Let $\Delta$ be a reflexive polytope
    and assume $\nu$ is a probability measure on $\partial\delta^\vee$ supported on the open faces. We will say that $\nu$ is stable if the optimal transport problem given by $(\partial\Delta,\mu_M), (\partial\Delta^\vee,\nu)$ and cost function $c(m,n)=-\langle m,n \rangle_{M,N}$ admits an optimal transport plan supported on 
    $$ \cup_{\sigma} \sigma^\circ\times \Star(n_\sigma)^\circ = \cup_{\tau} \Star(m_\tau)^\circ\times \tau^\circ.$$
\end{defi}

\begin{thm}\cite{AH}\label{thm:stability}
Let $\Delta$ be a reflexive polytope
and assume $\nu$ is a probability measure on $\partial\Delta^\vee$. Then there is a unique $\Psi:N_\R\rightarrow \R$ satisfying $|\sup_{m\in \Delta} \langle m,n \rangle -\Psi|$ and \eqref{eq:MAPolytope}
if and only if $\nu$ is stable. 
\end{thm}

When $\nu=\nu_N$, the stability condition above can be verified in cases where $\Delta$ have a lot of symmetries, for example the standard unit simplex (recovering Theorem~\ref{thm:symmetric}) and the unit cube). However, given a general reflexive polytope $\Delta$, it does not seem easy to check the stability of $\nu_N$. On the other hand, a necessary condition for stability of $\nu_N$ is given by the volume estimates 
\begin{equation} \label{eq:StructuralStability} \nu_N(\tau)\leq \mu(\Star(\mu_M)) \text{ and } \mu_M(\sigma)\leq \nu(\Star(n_\sigma)) \end{equation}
(see \cite[Lemma~19]{AH}) which are easily computed given $\Delta$. Computer aided computations show that out of the 4319 existing reflexive 3-dimensional polytopes, 1542 violate \eqref{eq:StructuralStability}, hence are not stable (see \cite[Table~2]{AH}). 

Conversely, a sufficient condition for existence of solutions (and hence stability) was given in \cite{LiToric}:
\begin{thm}\cite{LiToric}\label{thm:Li}
Let $\Delta$ be a reflexive polytope such that $\langle m,n \rangle \not=0$ for all vertices $m$ and $n$ of $\Delta$ and $\Delta^\vee$, respectively. Then there is a unique $\Psi:N_\R\rightarrow \R$ satisfying $|\sup_{m\in \Delta} \langle m,n \rangle -\Psi|$ and $$ \MA(\Psi) = \nu_N. $$ 
\end{thm}
Theorem~\ref{thm:Li} is striking since it applies to some polytopes with small symmetry groups. In particular, it applies when $\Delta$ is the polytope associated to $\P^3$ blown up in a torus invariant point. However, out of the 4319 existing reflexive 3-dimensional polytopes, only 238 satisfy this condition (see \cite[Table~2]{AH}). Together with \eqref{eq:StructuralStability}, this means that for the majority of reflexive polytopes $\Delta$ stability/existence of solutions is not known. 

By the breakthrough by Yang Li in \cite{LiFermat, LiSYZ}, the existence results in Theorem~\ref{thm:stability} and Theorem~\ref{thm:Li} have the following consequence:
\begin{cor}\label{cor:SYZ}
    Assume $\Delta$ is a reflexive Delzant Polytope and either $\nu_N$ is stable or the condition in Theorem~\ref{thm:Li} holds, then \eqref{eq:Family} satisfies the weak metric SYZ conjecture, i.e. for every $\delta>0$, the Calabi-Yau manifolds close to the central fiber admit a special Lagrangian torus fibration on a subset of normalized Calabi-Yau volume at least $1-\delta$ (see \cite{LiFermat} for details).
\end{cor}
\begin{rmk}
    When $\Delta$ is the standard unit simplex and $\{s_0+ts\}$ is the Fermat family, the weak metric SYZ conjecture was established (modulo passing to subsequence) by Yang Li in the seminal work \cite{LiFermat}. The requirement to pass to subsequence was eliminated and the scope extended to more general families of hypersurfaces of $\P^{d+1}$ in \cite{HJMM,PS}.
\end{rmk}
\begin{rmk}
In dimension 2, singular integral affine structures on the 2-sphere and solutions to Monge-Ampère equations where constructed by Loftin in \cite{Loftin}. 
\end{rmk}

\subsection{Open Problems}
\begin{itemize}
    \item Find sufficient conditions for stability (and thus existence of solutions) which are more general than the assumption in Theorem~\ref{thm:Li}. Find necessary conditions for stability/existence of solutions which are less general than \eqref{eq:StructuralStability}.
    \item Under what conditions are $\partial^c\Psi:\partial\Delta\rightarrow \partial\Delta^\vee$ a homeomorphism, and under what conditions does the inverse image of the $(d-1)$-skeleton of $\Delta$ under $\partial\Phi$ intersect the $(d-1)$-skeleton of $\Delta^\vee$ in a $(d-2)$-dimensional set? The latter, but not the former, is known for the standard unit simplex and the unit cube by \cite{AH}. 
    \item Given a Hessian manifold $(X,\nabla,g)$ with dual $(X,\nabla',g)$. Find a natural pairing of $X$ with itself, which generalize the pairings discussed in Section~\ref{sec:Pairings}. Incidentally, since a pairing induce a notion of $c$-convex functions when used as a cost function, this also addresses a question concerning good positivity notions for Hessian metrics asked in \cite{LiFermat,LiSurvey}.
\end{itemize}
In this exposition, we've intentionally avoided the question of regularity for the solutions to the Monge-Ampère equations. On compact Hessian manifolds, Pogorelov type counterexamples where ruled out by Caffarelli and Viaclovski in \cite{CafVia}. For the standard unit simplex and the unit cube, the solution is known to be smooth by \cite{AH}. However, for reflexive polytopes in general, the question of regularity is open. See for example \cite{Moo21, MR22} for some pathological behaviour in related settings.

\end{document}